\documentclass[Leqno, 12pt]{amsart}

\usepackage[paper=a4paper,left=2.95cm,right=2.95cm,top=3.7cm,bottom=3.7cm]{geometry}
\usepackage{float, graphicx, amssymb, dsfont}
\usepackage{subfig}
\usepackage{tikz}
\usetikzlibrary{arrows,shapes}

\newtheorem{thm}{Theorem}[section]
\newtheorem{cor}[thm]{Corollary}
\newtheorem{lem}[thm]{Lemma}
\newtheorem{prop}[thm]{Proposition}

\newtheorem{conj}[thm]{Conjecture}
\newtheorem*{thm*}{Theorem}
\theoremstyle{definition}

\theoremstyle{remark}
\newtheorem{rem}[thm]{Remark}
\newtheorem*{claim}{Claim}
\newtheorem*{ack*}{Acknowledgment}

\numberwithin{equation}{section}
\numberwithin{figure}{section}

\newcommand{\per}{\mathrm{per}}             

\newcommand{\pXtoY}{\phi : X \to Y}

\newcommand{\cA}{\mathcal{A}}       

\newcommand{\cE}{\mathcal{E}}
\newcommand{\cF}{\mathcal{F}}       

\newcommand{\cP}{\mathcal{P}}

\newcommand{\cS}{\mathcal{S}}
\newcommand{\cT}{\mathcal{T}}

\newcommand{\setN}{\mathbb{N}}
\newcommand{\setR}{\mathbb{R}}
\newcommand{\setZ}{\mathbb{Z}}

\newcommand{\A}{\mathcal{A}}    
\newcommand{\B}{\mathcal{B}}    
\newcommand{\F}{\mathcal{F}}    
\newcommand{\E}{\mathcal{E}}    

\newcommand{\X}{\mathsf{X}}

\newcommand{\ifff}{if and only if}

\newcommand{\fto}{finite-to-one}
\newcommand{\ito}{infinite-to-one}

\newcommand{\bir}{bi-resolving}

\newcommand{\bic}{bi-closing}

\newcommand{\SFT}{shift of finite type}
\newcommand{\SFTs}{shifts of finite type}
\newcommand{\rSFT}{irreducible shift of finite type}
\newcommand{\rSFTs}{irreducible shifts of finite type}
\newcommand{\mSFT}{mixing shift of finite type}
\newcommand{\mSFTs}{mixing shifts of finite type}
\newcommand{\Sofic}{sofic shift}
\newcommand{\rSofic}{irreducible sofic shift}
\newcommand{\mSofic}{mixing sofic shift}

\newcommand{\onto}{\xymatrix{\ar@{>>}[r]&}}
\newcommand{\da}[4]{\xymatrix{#1 \ar@<.5ex>[r]^{#2} \ar@<-.5ex>[r]_{#3} & #4}}

\newcommand{\dom}{\mathrm{dom}}

\begin{document}

\title[Decompositions of factors and embeddings]{Decompositions of factor codes and embeddings between shift spaces with unequal entropies}

\author[S. Hong]{Soonjo Hong}
\address{Centro de Modelamiento Matem\'atico \\
    Universidad de Chile \\
    Av. Blanco Encalada 2120, Piso 7 \\
    Santiago de Chile \\
    Chile}
\email{hsoonjo@dim.uchile.cl}

\author[U. Jung]{Uijin Jung}
\address{School of Mathematics \\
	Korea Institute for Advanced Study \\
	85 Hoegiro, Dongdaemun-gu \\
    Seoul 130-722 \\
	South Korea}
\email{uijin@kias.re.kr}

\author[I. Lee]{In-Je Lee}
\address{306-1 Bolli-dong \\
    Dalseo-gu \\
    Daegu 704-913 \\
    South Korea
}
\email{ijlee@kaist.ac.kr}


\maketitle

\begin{abstract}
    Given a factor code between \Sofic s $X$ and $Y$, there is a family of decompositions of the original code into factor codes such that the entropies of the intermediate subshifts arising from the decompositions are dense in the interval from the entropy of $Y$ to that of $X$. Furthermore, if $X$ is of finite type, we can choose those intermediate subshifts as shifts of finite type.
    In the second part of the paper, given an embedding from a shift space to an irreducible sofic shift, we characterize the set of the entropies of the intermediate subshifts arising from the decompositions of the given embedding into embeddings.
\end{abstract}

\section{Introduction}

Problems concerning decompositions of one code into several codes were considered in symbolic dynamics for coding purpose. The decomposition theorem of Williams states that every conjugacy between two \SFTs\ can be decomposed into a composition of splitting codes and amalgamation codes \cite{Wil74}.
Since closing codes form an important class of finite-to-one codes, Adler and Marcus \cite{AdlM} asked whether every finite-to-one factor code between irreducible \SFTs\ can be represented as a composition of closing codes. 
This is the case in the eventual category, but not true in general \cite{KitMT91}. Previous research on decompositions had concentrated on \fto\ factor codes \cite{Boy98, Tro95, Tro98}.
In particular, Boyle proved that up to conjugacy there are only finitely many decompositions of a finite-to-one factor code between \rSFTs\ into factor codes between \rSFTs\ \cite{Boy98}.
But so far, not much is known about decompositions of codes between shift spaces with unequal entropies, even in the category of \rSFTs.

\vspace{0.15cm}
In this paper, we consider decompositions of codes between shift spaces with unequal entropies. Let $\phi : X \to Y$ be a factor code between shift spaces. We are interested in the set $\cS(\phi)$ of all the entropies of the intermediate subshifts arising from the decompositions of $\phi$, defined by
\[
    \cS(\phi) = \{ h(\phi_1(X)) : \phi = \phi_2 \circ \phi_1 \text{ with } \phi_1 \text{ and } \phi_2 \text{ factor codes}\},
\]
where $h$ denotes topological entropy. Since the class of \SFTs\ is a fundamental and tractable class of shift spaces, if $X$ is of finite type we are also interested in the set $\cS_0(\phi)$ of all the entropies of the intermediate \SFTs\ arising from the decompositions, defined by
\[
    \begin{split}
    \cS_0(\phi) &= \{ h(\phi_1(X)) : \phi = \phi_2 \circ \phi_1 \text{ with } \phi_1 \text{ and } \phi_2 \text{ factor codes} \\
    &\hspace{2.65cm} \text{ and } \phi_1(X) \text{ of finite type} \}.
    \end{split}
\]

In \cite{BoyT84}, Boyle and Tuncel showed that if $\pXtoY$ is a factor code between irreducible shifts of finite type with unequal entropies then $h(X)$, and hence each element in $\cS_0(\phi)$ other than $h(Y)$, are limit points of $\cS_0(\phi)$. Thus up to conjugacy there are infinitely many decompositions of a factor code between \rSFTs\ with unequal entropies into factor codes. In \S 3, we extend this result as follows (see Theorem \ref{thm:decompose_factor_sofic_case} and Theorem \ref{thm:decompose_factor_SFT_case}).

\begin{thm}
    Let $\phi: X \to Y$ be a factor code between sofic shifts. Then $\cS(\phi)$ is dense in the interval $[h(Y),h(X)]$. If $X$ is of finite type, then $\cS_0(\phi)$ is also dense in $[h(Y),h(X)]$.
\end{thm}

We remark that by the result of Lindenstrauss \cite{LinE95}, if a factor map $\phi : X \to Y$ between shift spaces is given, then for any $h \in [h(Y),h(X)]$ we can find a topological dynamical system $(Z,T)$ such that there are surjective homomorphisms $\phi_1 : (X,\sigma) \to (Z,T)$ and $\phi_2 : (Z,T) \to (Y,\sigma)$ with $\phi = \phi_2 \circ \phi_1$ and $h(T) = h$. However, even when $X$ and $Y$ are \SFTs, the constructed system $Z$ is far from a shift space. Indeed, if $X$ is a sofic shift then $\cS(\phi)$ must be a countable set since factors of $X$ are sofic and there are countably many real numbers that can appear as the entropies of sofic shifts, namely, rational multiples of logarithms of Perron numbers.

\vspace{0.1cm}
The existence of factor codes and that of embeddings (especially between \rSFTs\ with different entropies)
are closely related in symbolic dynamics \cite{Boy83, Kri82}. In this viewpoint, we turn to decompositions of embeddings in \S 4. Let $\phi : X \to Y$ be an embedding into an irreducible shift space. We are interested in the set $\cT(\phi)$ of all the entropies of the intermediate shift spaces arising from decomposing $\phi$ into embeddings, defined by
\[
    \begin{split}
    \cT(\phi) &= \{ h(\dom(\phi_2)) : \phi = \phi_2 \circ \phi_1 \text{ with } \phi_1 \text{ and } \phi_2 \text{ embeddings} \\
    &\hspace{3.00cm}  \text{ and } \dom(\phi_2) \text{ irreducible} \},
    \end{split}
\]
where $\dom(\phi_2)$ is the domain of $\phi_2$.
We impose an irreducibility condition on the domain of $\phi_2$ since otherwise the problem of finding a decomposition $\phi = \phi_2 \circ \phi_1$ into embeddings with $h(\dom(\phi_2))=h$ is reduced to that of finding a subshift of $Y$ with entropy $h$, and this problem is completely understood in symbolic dynamics.
Define $\cT_0(\phi)$ (resp. $\cT_1(\phi)$) as in $\cT(\phi)$ with additional condition that $\dom(\phi_2)$ is of finite type (resp. sofic).
To exclude a subtle problem at the point $h(X)$, we define $\cT'(\phi) = \cT(\phi) \setminus \{ h(X) \}$ and similarly for $\cT'_0(\phi)$ and $\cT'_1(\phi)$.
It is well known that if $Y$ is a \mSFT\ then $\cT_0(\phi)$ is dense in $[h(X),h(Y)]$ \cite{DGS}. We refine this result and characterize these sets as follows (see Corollary \ref{cor:cT_sofic}, Corollary \ref{cor:cT_0_SFT} and Corollary \ref{cor:cT_1_sofic}). In what follows, $\cP$ (resp. $\cP^w$) denotes the set of all real numbers $h \geq 0$ such that $e^h$ a Perron number (resp. weak Perron number).

\begin{thm} \label{thm:embedding_introduction}
    Let $\phi: X \to Y$ be an embedding from a shift space $X$ into an irreducible sofic shift $Y$. Then we have
    \[
        \cT'(\phi) = (h(X),h(Y)] \text{ and } \cT'_1(\phi) = (h(X),h(Y)] \cap \cP^w.
    \]
    If $X$ and $Y$ are \rSFTs\ with periods $p$ and $q$, respectively, then
    \[
        \cT_0(\phi) = [h(X),h(Y)] \cap \{ h \in \setR : r\cdot h \in \cP \text{ for some } r \in \setN \text{ with } q|r \text { and } r|p \}.
    \]
\end{thm}
We also present characterizations of the sets $\cT$, $\cT_0$ and $\cT_1$ for other cases in \S 4.

\vspace{0.15cm}
\section{Background}
We introduce some terminology and known results. For further details, see \cite{LM}.
A \emph{shift space} (or \emph{subshift}) is a closed $\sigma$-invariant subset of a full shift over some finite set of symbols. For a subshift $X$, denote by $\B_n(X)$ the set of all words of length $n$ appearing in the points of $X$ and $\B(X) = \bigcup_{n \geq 0} \B_n(X)$. Also let $\cA_X = \B_1(X)$. For a finite set $\cA$ and $l \in \setN$, denote by $\cA^l$ the set of all words of length $l$ over $\cA$ and let $\cA^* = \bigcup_{l \geq 0} \cA^l$.

A subshift $X$ is called \emph{nonwandering} if for all $u \in \B(X)$, there is a word $w$ with $uwu \in \B(X)$. It is called \emph{irreducible} if for all $u, v \in \B(X)$, there is a word $w$ with $uwv \in \B(X)$. It is called \emph{mixing} if for all $u, v \in \B(X)$, there is an integer $N \in \setN$ such that whenever $n \geq N$, we can find $w \in \B_n(X)$ with $uwv \in \B(X)$. The \emph{period} of $X$, denoted by $\per(X)$, is the greatest common divisor of periods of all periodic points of $X$. The \emph{entropy} of a shift space is defined by $h(X) = \lim_{n \to \infty} (1/n) \log |\B_n(X)|$, which equals the topological entropy of $(X,\sigma)$ as a dynamical system. By the variational principle, topological entropy is concentrated on the nonwandering set in the sense that if $Z$ is the maximal nonwandering set of $X$ then $h(Z) = h(X)$.

A \emph{code} $\pXtoY$ is a continuous $\sigma$-commuting map between shift spaces. It is called a \emph{factor code} (resp. an \emph{embedding}) if it is surjective (resp. injective). 
Every code can be recoded to be a 1-block code, i.e., a code for which $x_0$ determines $\phi(x)_0$.

Let $X$ be a subshift over a finite set $\cA$. Then one can find a set $\cF_X$ of \emph{forbidden words} so that $X$ is the set of all sequences in $\cA^\setZ$ which do not contain any words in $\cF_X$. If there is such a set $\cF_X$ of words all of which have length $k+1$ with $k \geq 0$, then $X$ is called a (\emph{$k$-step}) \emph{\SFT}. Equivalently, $X$ is $k$-step if every word $v \in \B_k(X)$ is a \emph{synchronizing word} for $X$, i.e., if $uv$ and $vw$ are in $\B(X)$ then we have $uvw \in \B(X)$.
A \emph{sofic shift} is a factor of a \SFT. Given a sofic shift $X$, there is a factor code $\pi : \widetilde X \to X$ where $\widetilde X$ is of finite type with $h(\widetilde X) = h(X)$. Furthermore, if $X$ is irreducible (resp. mixing), then $\widetilde X$ can be chosen to be irreducible (resp. mixing).
Every sofic shift $X$ contains a family of shifts of finite type whose entropies are dense in $[0,h(X)]$ \cite{Mar85}.

A real number $\geq 1$ is called a \emph{Perron number} if it is an algebraic integer which strictly dominates all the other algebraic conjugates. A real number $\lambda$ is a Perron number \ifff\ there is a \mSFT\ $X$ with $h(X) = \log \lambda$. 
Similarly, there is an \rSFT\ $X$ of period $p \in \setN$ with $h(X) = \log \lambda$ \ifff\ $\lambda^p$ is a Perron number. A real number $\lambda \geq 1$ is called a \emph{weak Perron number} if $\lambda^p$ is a Perron number for some $p \in \setN$.

\vspace{0.15cm}
\section{Decompositions of factor codes}

In \cite[Proposition 7.3]{BoyT84}, Boyle and Tuncel considered the notion of a magic diamond to investigate properties of Markovian codes, and obtained the following result.

\begin{prop}\cite{BoyT84} \label{prop:BoyT}
    Let $\pXtoY$ be a factor code between irreducible shifts of finite type. Then every element in $\cS_0(\phi) \setminus \{ h(Y) \}$ is a limit point of $\cS_0(\phi)$ from the left.
\end{prop}

In this section, we extend this result and show that under the condition of Proposition \ref{prop:BoyT}, $\cS_0(\phi)$ is indeed dense in the interval $[h(Y),h(X)]$. We first consider the sofic case in which we do not require intermediate subshifts to be of finite type. This result can also be obtained from the result of Lindenstrauss (see Remark \ref{rem:decompose_factor_sofic_case} (2)).

\begin{thm}\label{thm:decompose_factor_sofic_case}
    Let $\pXtoY$ be a factor code between sofic shifts. Then $\cS(\phi)$ is dense in $[h(Y),h(X)]$.
\end{thm}

\begin{proof}
    Since the case where $h(X) = h(Y)$ is clear, assume that $h(X) > h(Y)$. Let $h \in (h(Y), h(X))$ and $\epsilon > 0$. We claim that there is a decomposition $\phi = \phi_2 \circ \phi_1$ of $\phi$ into factor codes such that $h(\phi_1(X))$ is $\epsilon$-close to $h$. First, find a shift of finite type $Z \subsetneq X$ with $h(Z) > h(Y)$ and $|h(Z) - h| < \epsilon / 2$. 
    By passing to higher block shifts and renaming symbols, we may assume that \begin{enumerate}
        \item[(a)] $\phi$ is a $1$-block code,
        \item[(b)] $Z$ is 1-step,
        \item[(c)] $\A_Z$ and $\A_Y$ are disjoint,
        \item[(d)] if $a, b \in \A_Z$ and $ab \in \B(X)$, then $ab \in \B(Z)$.
    \end{enumerate}
    Let $\A = \A_Z \cup \A_Y$. Define a 1-block code $\theta : X \to \A^\setZ$ by
    \[
        {\theta(x)_i} =
            \begin{cases}
                \phi(x_i)   & \text{if $x_i \notin \A_Z$}   \\
                x_i         & \text{if $x_i \in \A_Z$}   \\
            \end{cases}
    \]
    and let $\tilde Z_0 = \theta(X)$. By the condition (d), each point in $\tilde Z_0$ is uniquely factored as (possibly infinite) $Z$-words and $Y$-words (hence if $u \in \cA_Z^*$ occurs in $\tilde Z_0$, then $u \in \B(\tilde Z_0)$).

    Also define a (3-block) code $\alpha : \A^\setZ \to \A^\setZ$ by
    \[
        {\alpha(x)_i} =
            \begin{cases}
                \phi(x_i)   &   \text{if $x_i \in \A_Z$ and if $x_{i-1} \in \A_Y$ or $x_{i+1} \in \A_Y$} \\
                x_i         &   \text{otherwise}                                    \\
            \end{cases}
    \]
    and, define $\tilde Z_n = \alpha(\tilde Z_{n-1})$ inductively for $n \in \setN$. Note that for each $n \geq 1$, $\tilde Z_n$ is a sofic shift whose language contains no word of the form $a w b$ where $a, b \in \A_Z$ and $w \in \B(Y)$ with $ 1 \leq |w| \leq 2n$. Also $Z \subset \tilde Z_n$ for each $n \in \setN$. Finally, for each $n \in \setN$ define a shift space $\hat Z_n$ over $\A$ obtained by forbidding the set of words
    \[
        \F_n = ( \F_Y \cup \F_{Z} ) \cup \bigcup_{j=1}^{2n} \A_Z \A_Y^j \A_Z,
    \]
    where $\F_Y$ (resp. $\F_{Z}$) is a set of forbidden words of $Y$ (resp. $Z$) over $\cA_Y$ (resp. $\cA_Z$).
    By (d), one can see that $\tilde Z_n \subset \hat Z_n$ for each $n \in \setN$. By letting $\hat Z = \bigcap_{n \in \setN} \hat Z_n$, we have $h(\hat Z) = \lim_n h(\hat Z_n)$ since $\{ \hat Z_n \}_{n \in \setN}$ is a decreasing sequence of shift spaces. Since $\hat Z$ contains no word of the form $a w b$ with $a, b \in \A_Z$ and $w \in \B(Y)$, it follows that
    \[ \B(\hat Z) = \{ u v w : u , w \in \B(Y) \text{ and } v \in \B(Z) \}. \]
    So the nonwandering set of $\hat Z$ is contained in $Y \cup Z$. Since the entropy is concentrated on the nonwandering set and $h(Y) < h(Z)$, we have
    \[ h(Z) \leq \lim_n h(\hat Z_n) = h(\hat Z) \leq h(Z).\]

    Now take $N \in \setN$ such that $|h(\tilde Z_N) - h(Z)| < \epsilon / 2$ and let $\phi_1 = \alpha^N \circ \theta : X \to \tilde Z_N$. Also define a $1$-block code $\phi_2 : \tilde Z_N \to Y$ by letting $\phi_2(z)_i = z_i$ if $z_i \in \A_Y$ and $\phi_2(z)_i = \phi(z_i)$ if $z_i \in \A_Z$. Then $\phi = \phi_2 \circ \phi_1$ and $\phi_2$ is indeed a factor code. The following inequality completes the proof:
    \[
        |h(\phi_1(X)) - h| = |h(\tilde Z_N) - h| \leq |h(\tilde Z_N) - h(Z)| + |h(Z) - h| < \epsilon.
    \]
\end{proof}

\begin{rem}\label{rem:decompose_factor_sofic_case}
    (1) The proof of Theorem \ref{thm:decompose_factor_sofic_case} also applies to the case where $Y$ is a shift space and $X$ is a shift space such that there are shifts of finite type $X_n \subset X$ with $\lim_{n \to \infty} h(X_n) = h(X)$ (for example, $X$ may be an \emph{almost specified shift} \cite{Jung11})

    (2) By corollary of Lindenstrauss' result stated in \S 1, the conclusion of Theorem \ref{thm:decompose_factor_sofic_case} holds for arbitrary factor code $\pXtoY$ between shift spaces. To see this, suppose that we have a factorization $\phi = \phi_2 \circ \phi_1$ where $\phi_1 : X \to Z$ and $\phi_2 : Z \to Y$ are surjective homomorphisms as in \cite[\S 4]{LinE95}. Then $Z$ is zero-dimensional and can be presented as an inverse limit of shift spaces $Z_n$. By uniform continuity of $\phi_2$, there exists $N \in \setN$ such that for every $n \geq N$, if $\pi_n$ is a projection from $Z$ onto $Z_n$, then there is $\gamma_n$ such that $\phi_2 = \gamma_n \circ \pi_n$. So $\phi$ can be factored through the shift space $Z_n$. Since $h(Z) = \lim h(Z_n)$, it follows that $\cS(\phi)$ is dense in $[h(Y),h(X)]$.

\end{rem}

\vspace{0.2cm}
We now prove the case of \SFTs. The heart of Theorem \ref{thm:decompose_factor_SFT_case} lies in the following proposition. For a shift space $X$ and $n \in \setN$, denote by $X^{[n]}$ the \emph{$n$-th higher block shift} of $X$ \cite[\S 1.4]{LM}.

\begin{prop}\label{prop:decompose_factor_SFT_case_simplified}
    Let $\pXtoY$ be a factor code from a \SFT\ $X$ with $h(X) > h(Y)$. Then for any $\epsilon > 0$, there exists a decomposition $\phi = \phi_2 \circ \phi_1$ of $\phi$ into factor codes such that $|h(\phi_1(X)) - h(Y)| < \epsilon$ and $\phi_1(X)$ is of finite type.
\end{prop}
\begin{proof}
    We may assume that $X$ is $1$-step, $\phi$ is $1$-block, and that $\A_X$ and $\A_Y$ are disjoint. Let $\cF_Y$ be a set of forbidden words of $Y$ over $\cA_Y$.

    For each $n \in \setN$ and a symbol $a$ not in $\cA_Y$, define the set $\cF_n(a)$ of words in $( \A_Y \cup \{ a \})^*$ by
    \[
        \cF_n(a) = \cF_Y \cup \Bigl( \bigcup_{i=0}^{[\frac{n}{4}] - 1} a \A_Y^i a \Bigr) ,
    \]
    (by convention we let $\cF_n(a) = \cF_Y$ for $n < 4$) and the shift space $\X_{\cF_n(a)}$ over the alphabet $\A_Y \cup \{a\}$ by forbidding the set $\cF_n(a)$. As $\cF_n(a)$ increases with $n$, the shift space $\X_{\cF_n(a)}$ decreases with $n$. Since $Y \subset \X_{\cF_n(a)}$ for each $n \in \setN$, and the nonwandering set of $\bigcap_{n \in \setN} \X_{\cF_n(a)}$ is contained in $Y$, we have
    \[ h(Y) \leq \lim_n h(\X_{\cF_n(a)}) = h \Bigl( \bigcap_{n \in \setN} \X_{\cF_n(a)} \Bigr) \leq h(Y). \]

    Choose $n \in \setN$ such that $h(\X_{\cF_n(a)}) < h(Y) + \epsilon$. Then indeed we have $h(\X_{\cF_n(\bar a)}) < h(Y) + \epsilon$ for any symbol $\bar a$ not in $\cA_Y$. We may assume that $n$ is a multiple of $4$. Since $X$ is 1-step, so is $X^{[n]}$ and each symbol of $X^{[n]}$ is a synchronizing word for $X^{[n]}$. Let $\cA = \A_{X^{[n]}} \cup \A_Y$. In the remainder of the proof, we will regard an element in $\B_1(X^{[n]})$ as a symbol of $X^{[n]}$ and also as a word in $\B_n(X)$, depending on the context.

    Now we partition $\cA_{X^{[n]}} = \B_1(X^{[n]})$ into disjoint sets $\cE_1, \cE_2, \cdots, \cE_N$ so that $\cE_N$ contains all $X$-words of length $n$ each of which has a ``large" self-overlap. The precise definitions of $\cE_j$'s are as follows:
    \begin{enumerate}
        \item
            For $j = 1, \cdots, N-1$, each set $\cE_j$ contains exactly one symbol $ a^{(j)}$ which, as an $X$-word, does not overlap itself, or overlaps itself only after a shift of more than ${n/4}$ symbols to the right. We require that $\bigcup_{j=1}^{N-1} \cE_j$ exhausts all such symbols.
        \item
            $\cE_N = \B_1(X^{[n]}) \setminus \bigcup_{j=1}^{N-1} \cE_j$.
            Note that, each $b \in \cE_N$, as an $X$-word, can overlap itself after a shift of at most ${n/4}$ symbols to the right (hence each $b$ overlaps itself in at least ${3n/4}$ symbols).
    \end{enumerate}

    Let $m \in \setN$. Define a ($2m+1$)-block code $\phi^{(m)}$ from $X^{[n]}$ to $\A^\setZ$ by
    \[
        {\phi^{(m)}(x)_0} =
            \begin{cases}
                x_0   &   \text{if $x_i \in \bigcup_{k \geq j} \cE_k$ for all $i = -m, \cdots, m$}     \\
                \phi((x_0)_1)         &   \text{otherwise}                                    \\
            \end{cases}
    \]
    where $j$ is the unique number with $x_0 \in \cE_j$ and $(x_0)_1$ is the first symbol of $x_0$ as an $X$-word of length $n$.
    (Intuitively, consider the case $n=1$. Then $\cE_i$'s give an order on the set $\cA_X$ defined by $a < b$ if $a \in \E_i$ and $b \in \E_j$ with $i < j$.
    If we do not see a symbol `less than' $x_0$ in the $[-m,m]$ neighborhood of $x_0$, we leave $x_0$ unchanged. If there is a symbol `less than' $x_0$ in the neighborhood, we map $x_0$ using $\phi$. The idea of $\phi^{(m)}$ is to `leave' $X$ symbols unchanged syndetically for each point in $\phi^{(m)}(X)$.)

    \begin{claim}
        The image $Z_m = \phi^{(m)}(X^{[n]})$ is a $(2mN+1)$-step shift of finite type.
    \end{claim}
    \begin{proof}
        First, we prove that any $Z_m$-word of length ($2mN + 1$) contains a symbol of $X^{[n]}$. Let $w \in \B_{2mN+1}(Z_m)$. Take $z \in Z_m$ with $z_{[-mN,mN]} = w$ and $x \in X^{[n]}$ with $\phi^{(m)}(x) = z$. Also for each $i = -mN, \cdots, mN$, let $j_i$ be the unique index such that $x_i \in \cE_{j_i}$. If $z_0 \in \A_{X^{[n]}}$, we are done. If not, then there exists $i_1 \in [-m,m]$ such that $j_{i_1} < j_0$. If $z_{i_1} \in \A_{X^{[n]}}$, we are done. If not, then there exists $i_2 \in [i_1 - m, i_1 + m]$ such that $j_{i_2} < j_{i_1}$ and so on. This process eventually terminates after at most ($N-1$) steps and there is $i \in [-(N-1)m,(N-1)m]$ such that $j_i \leq j_k$ for $k = i-m, \cdots, i+m$. Then $z_i = x_i \in \A_{X^{[n]}}$, as desired.

        Second, we show that each symbol in $\A_{X^{[n]}}$ is also a synchronizing symbol for $Z_m$. To show this, let $ua, av \in \B(Z_m)$ with $a \in \A_{X^{[n]}}$.
        Take $x, y \in X^{[n]}$ with $\phi^{(m)}(x)_{[-|u|,0]} = ua$ and $\phi^{(m)}(y)_{[0,|v|]} = av$. Since $a \in \A_{X^{[n]}}$ we have $x_0 = y_0 = a$. Then $z = x_{(-\infty,0)}y_{[0,\infty)} \in X^{[n]}$ since $a$ is a synchronizing word for $X^{[n]}$. We claim that $\phi^{(m)}(z)_{[-|u|,|v|]} = uav$. Since $\phi^{(m)}$ has memory and anticipation $m$, we have $\phi^{(m)}(z)_i = \phi^{(m)}(y)_i$ for $i \geq m$. Let $j_0$ be the unique index with $a \in \cE_{j_0}$. Since $ua, av \in \B(Z_m)$, we have $x_k, y_k \in \bigcup_{j \geq j_0} \cE_j$ for each $k = -m, \cdots, m$ and hence $\phi^{(m)}(z)_0 = a$ (see Figure \ref{fig:image_of_z_0}).

        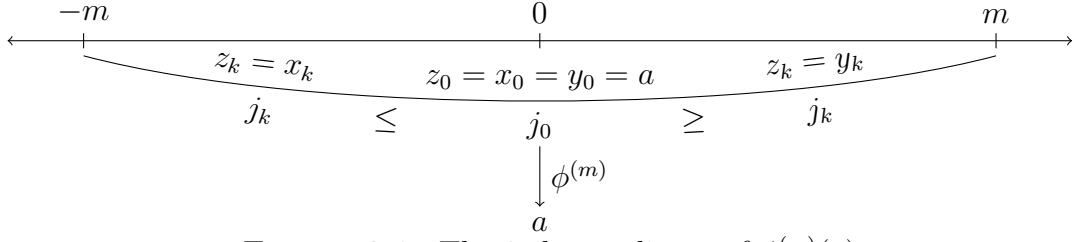
\begin{figure}[h]
            \vspace{-0.5cm}
            \center{
                \begin{tikzpicture}
                    \draw[<->] (0,0) -- (14,0);
                    \draw (1,0.1) -- (1,-0.1) node[pos=0.05, above] {$-m$};
                    \draw (7,0.1) -- (7,-0.1) node[pos=0.05, above] {$0$};
                    \draw (13,0.1) -- (13,-0.1) node[pos=0.05, above] {$m$};
                    \draw (1,-0.2) .. controls (4,-1) and (10,-1) .. (13,-0.2)
                        node[pos=0.22, sloped, above] {$z_k=x_k$} node[pos=0.22, sloped, below] {$j_k$}
                        node[pos=0.35, sloped, below] {$\le$}
                        node[pos=0.5, sloped, above] {$z_0=x_0=y_0=a$} node[pos=0.5, sloped, below] {$j_0$}
                        node[pos=0.65, sloped, below] {$\ge$}
                        node[pos=0.78, sloped, above] {$z_k=y_k$} node[pos=0.78, sloped, below] {$j_k$};
                    \draw [->](7,-1.4) -- (7,-2.2) node[pos=0.5, right] {$\phi^{(m)}$} node[below] {$a$};
                \end{tikzpicture}
            }
            \vspace{-0.5cm}
            \caption{The $0$-th coordinate of $\phi^{(m)}(z)$}\label{fig:image_of_z_0}
        \end{figure}
        \vspace{-0.3cm}

        Now given $i = 1, \cdots, m-1$, since $z_i = y_i \in \bigcup_{j \geq j_0} \cE_{j}$ we have the following cases (see Figure \ref{fig:image_of_z_i}):
        \begin{enumerate}
            \item
                If $z_i \in \bigcup_{j > j_0} \cE_{j} $, then $\phi^{(m)}(z)_i = \phi((z_i)_1) = \phi((y_i)_1) = \phi^{(m)}(y)_i$.
            \item
                If $z_i \in \cE_{j_0}$ and $z_k \in \bigcup_{j \geq j_0} \cE_{j}$ for all $k \in [m+1,m+i]$, then
                \[ \phi^{(m)}(z)_i = z_i = y_i = \phi^{(m)}(y)_i. \]
            \item
                If $z_i \in \cE_{j_0}$ and $z_k \in \bigcup_{j < j_0} \cE_{j}$ for some $k \in  [m+1,m+i]$, then
                \[ \phi^{(m)}(z)_i = \phi((z_i)_1) = \phi((y_i)_1) = \phi^{(m)}(y)_i.\]
        \end{enumerate}
        \begin{figure}[h]
            \center{
                \begin{tikzpicture}
                    \draw[<->] (0,0) -- (9,0);
                    \draw (1,0.1) -- (1,-0.1) node[pos=0.05, above] {$i-m$};
                    \draw (2.5,0.1) -- (2.5,-0.1) node[pos=0.05, above] {$0$};
                    \draw (4.5,0.1) -- (4.5,-0.1) node[pos=0.05, above] {$i$};
                    \draw (6.0,0.1) -- (6.0,-0.1) node[pos=0.05, above] {$m$};
                    \draw (8,0.1) -- (8,-0.1) node[pos=0.05, above] {$i+m$};
                    \draw (1,-0.2) .. controls (3,-1) and (6,-1) .. (8,-0.2)
                        node[pos=0.22, sloped, above] {$z_0=y_0$} node[pos=0.22, sloped, below] {$j_0$}
                        node[pos=0.35, sloped, below] {$<$}
                        node[pos=0.5, sloped, above] {$z_i=y_i$} node[pos=0.5, sloped, below] {$j_i$};
                    \draw [->](4.5,-1.4) -- (4.5,-2.2) node[pos=0.5, right] {$\phi^{(m)}$} node[below] {$\phi((z_i)_1)=\phi((y_i)_1)$};
                \end{tikzpicture}
            }
            \center{
                \begin{tikzpicture}
                    \draw[<->] (0,0) -- (9,0);
                    \draw (1,0.1) -- (1,-0.1) node[pos=0.05, above] {$i-m$};
                    \draw (2.5,0.1) -- (2.5,-0.1) node[pos=0.05, above] {$0$};
                    \draw (4.5,0.1) -- (4.5,-0.1) node[pos=0.05, above] {$i$};
                    \draw (6.0,0.1) -- (6.0,-0.1) node[pos=0.05, above] {$m$};
                    \draw (8,0.1) -- (8,-0.1) node[pos=0.05, above] {$i+m$};
                    \draw (1,-0.2) .. controls (3,-1) and (6,-1) .. (8,-0.2)
                        node[pos=0.22, sloped, above] {$z_0=y_0$} node[pos=0.22, sloped, below] {$j_0$}
                        node[pos=0.35, sloped, below] {$=$}
                        node[pos=0.5, sloped, above] {$z_i=y_i$} node[pos=0.5, sloped, below] {$j_i$}
                        node[pos=0.68, sloped, below] {$\le$}
                        node[pos=0.82, sloped, above] {$\forall z_k=y_k$} node[pos=0.85, sloped, below] {$j_k$};
                    \draw [->](4.5,-1.4) -- (4.5,-2.2) node[pos=0.5, right] {$\phi^{(m)}$} node[below] {$z_i=y_i$};
                \end{tikzpicture}
            }
            \center{
                \begin{tikzpicture}
                    \draw[<->] (0,0) -- (9,0);
                    \draw (1,0.1) -- (1,-0.1) node[pos=0.05, above] {$i-m$};
                    \draw (2.5,0.1) -- (2.5,-0.1) node[pos=0.05, above] {$0$};
                    \draw (4.5,0.1) -- (4.5,-0.1) node[pos=0.05, above] {$i$};
                    \draw (6.0,0.1) -- (6.0,-0.1) node[pos=0.05, above] {$m$};
                    \draw (8,0.1) -- (8,-0.1) node[pos=0.05, above] {$i+m$};
                    \draw (1,-0.2) .. controls (3,-1) and (6,-1) .. (8,-0.2)
                        node[pos=0.22, sloped, above] {$z_0=y_0$} node[pos=0.22, sloped, below] {$j_0$}
                        node[pos=0.35, sloped, below] {$=$}
                        node[pos=0.5, sloped, above] {$z_i=y_i$} node[pos=0.5, sloped, below] {$j_i$}
                        node[pos=0.68, sloped, below] {$>$}
                        node[pos=0.82, sloped, above] {$\exists z_k=y_k$} node[pos=0.85, sloped, below] {$j_k$};
                    \draw [->](4.5,-1.4) -- (4.5,-2.2) node[pos=0.5, right] {$\phi^{(m)}$} node[below] {$\phi((z_i)_1)=\phi((y_i)_1)$};
                \end{tikzpicture}
            }

            \vspace{-0.3cm}
            \caption{The $i$-th coordinate of $\phi^{(m)}(z)$}\label{fig:image_of_z_i}
            \vspace{-0.1cm}
        \end{figure}
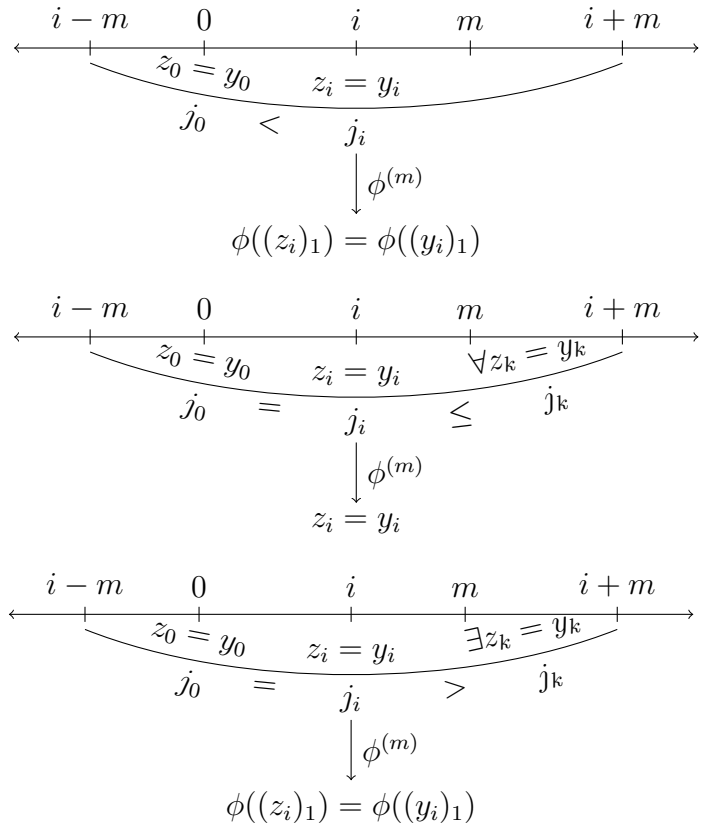
        Hence $\phi^{(m)}(z)_i = \phi^{(m)}(y)_i$ for $i \geq 0$. Similarly we have $\phi^{(m)}(z)_i = \phi^{(m)}(x)_i$ for $i \leq 0$, so $\phi^{(m)}(z)_{[-|u|,|v|]} = uav$, as desired.

        Since every word in $Z_m$ of length ($2mN + 1$) contains a synchronizing word for $Z_m$, it follows that $Z_m$ is a $(2mN+1)$-step \SFT.
    \end{proof}

    Now we give an upper bound of the entropy of $Z_m = \phi^{(m)}(X^{[n]})$. We claim that $Z_m \subset \X_{\tilde \cF_m}$, where $\X_{\tilde \cF_m}$ is the subshift over $\cA$ obtained by forbidding the set of words
    \[ \tilde \cF_m = \cF_Y \cup
        \Bigl( \bigcup_{\substack{i, j = 1 \\ i \neq j}}^{N} \bigcup_{k=0}^{m-1} \cE_i \cA_Y^k \cE_j \Bigr) \cup
        \Bigl( \bigcup_{i=1}^{N-1} \bigcup_{k=0}^{\frac{n}{4} - 1} \cE_i \cA_Y^{k} \cE_i \Bigr) \cup
        \Bigl( \bigcup_{k=1}^{m-1} \cE_N \cA_Y^k \cE_N \Bigr). \]

    To prove the claim, suppose that $z \in Z_m$.
    \begin{enumerate}
        \item[(a)]
            Since any word in $\A_Y^*$ appearing in $z$ comes from applying $\phi$ on an $X$-word of the same length, it follows that a word in $\cF_Y$ cannot occur in $z$.
        \item[(b)]
            Suppose that a word of the form $a^{(i)} w a^{(j)}$, with $a^{(i)} \in \cE_i$, $a^{(j)} \in \cE_j$, $i \neq j$ and $|w| < m$ occurs in $z$. If $i > j$, then the definition of $\phi^{(m)}$ forces $a^{(i)}$ to be mapped by $\phi^{(m)}$ to $\phi((a^{(i)})_1)$ (note that $\phi^{(m)}$ has memory and anticipation $m$), a contradiction. Similarly we have a contradiction if $i < j$.
        \item[(c)]
            If a word of the form $a^{(i)} w a^{(i)}$ with $|w| < n/4$ and $i < N$ occurs in $z$, then the last ($n-|w|-1$) symbols of $a^{(i)}$ (as an $X$-word) must equal the first ($n-|w|-1$) symbols of $a^{(i)}$ because of the overlapping property of $X^{[n]}$. But this is impossible, since each $a^{(i)}, i = 1, \cdots, N-1$ does not overlap itself in $\geq 3n/4$ symbols (as an $X$-word).
        \item[(d)]
            Suppose that a word $a w b$ with $a, b \in \cE_N$ and $w \in \bigcup_{k=1}^{m-1} \cA_Y^k$ occurs in $z$. Let $\gamma$ be a preimage of $awb$ under $\phi^{(m)}$. Then $\gamma$ must contain a subword of the form $a u b \in \B(X^{[n]})$ in the center. However if $u$ contains a symbol in $\bigcup_{k=1}^{N-1} \cE_k$, then $a$ and $b$ must be mapped by $\phi^{(m)}$ to the symbols $\phi(a_1)$ and $\phi(b_1)$, respectively. Also, if $u$ consists of symbols in $\cE_N$, then since $u$ is mapped to $w$ by $\phi^{(m)}$, some symbol in $\bigcup_{k=1}^{N-1} \cE_k$ appears to the left of $a$ or to the right of $b$ within distance $m$ in $\gamma$. Thus at least one of $a$ and $b$ also has to be mapped by $\phi^{(m)}$ to $\phi(a_1)$ or $\phi(b_1)$, which is a contradiction.
    \end{enumerate}
    Thus $z \in \X_{\tilde \cF_m}$ and the claim holds.

    \vspace{0.1cm}
    Now $\bigcap_m \X_{\tilde \cF_m}$ equals the shift space $\X_{\tilde \cF}$ defined by forbidding the set of words $\tilde \cF = \bigcup_{m \in \setN} \tilde \cF_m$. Note that in $\X_{\tilde \cF}$ there is no transition of the form $\cE_i \to \cE_j$ or $\cE_i \to Y \to \cE_j$ with $i \neq j$, and also no transition of the form $\cE_N \to Y \to \cE_N$.
    Thus the nonwandering set of $\X_{\tilde \cF}$ is contained in the set
    \[ Y \cup \big( \bigcup_{i=1}^{N-1} \X_{\cF_n(a^{(i)})} \big) \cup \big( X^{[n]} \cap (\cE_N)^\setZ \big). \]

    By the choice of $n$, we have $h(\X_{\cF_n(a^{(i)})}) < h(Y) + \epsilon$ for each $i = 1, \cdots, N-1$. Also, the shift space $X^{[n]} \cap (\cE_N)^\setZ$ consists only of periodic points. To see this, let $x \in X^{[n]} \cap (\cE_N)^\setZ$. Then each $x_i$ has a self-overlap more than $3n/4$ symbols and for each $i \in \setZ$, $x_i$ and $x_{i+1}$ overlap progressively. It follows that $x_0$ determines the whole point $x$ and $x$ must be periodic.

    Hence it follows that
    \[ \lim_m h(\X_{\tilde \cF_m}) = h(\X_{\tilde \cF}) \leq \max_{1 \leq i < N} \{ h(\X_{\cF_n(a^{(i)})}), h(Y), h \bigl(X^{[n]} \cap ( \cE_N)^\setZ \bigr) \} < h(Y) + \epsilon \]
    as desired. Take $M \in \setN$ so that $h(\X_{\tilde \cF_M}) < h(Y) + \epsilon$ and let $\phi_1 = \phi^{(M)} \circ \beta$, where $\beta : X \to X^{[n]}$ is the $n$-th higher block code. Finally define a 1-block code $\phi_2 : Z_M \to Y$ by sending all symbols $a \in \cA_{X^{[n]}}$ to $\phi((a)_1)$ and $a \in \cA_Y$ to $a$. Then $\phi_1$ and $\phi_2$ satisfy all the conditions.
\end{proof}

Combining Theorem \ref{thm:decompose_factor_sofic_case} and Proposition \ref{prop:decompose_factor_SFT_case_simplified}, we obtain the main theorem in this section.

\begin{thm}\label{thm:decompose_factor_SFT_case}
    Let $\pXtoY$ be a factor code with $X$ of finite type. Then $\cS_0(\phi)$ is dense in $[h(Y),h(X)]$.
\end{thm}
\begin{proof}
    The case $h(X) = h(Y)$ is clear. Let $h \in (h(Y),h(X))$ and $\epsilon > 0$. Then by Theorem \ref{thm:decompose_factor_sofic_case}, there is a decomposition $\phi = \xi \circ \psi$ such that $|h(\psi(X)) - h| < \epsilon/2$. Now by Proposition \ref{prop:decompose_factor_SFT_case_simplified}, there is a decomposition $\psi = \eta \circ \phi_1$ such that $\phi_1(X)$ is of finite type and $|h(\phi_1(X)) - h(\psi(X))| < \epsilon/2$. By letting $\phi_2 = \xi \circ \eta$, we are done.
\end{proof}

There are some cases in which we can characterize the sets $\cS(\phi)$ and $\cS_0(\phi)$. For example, if $X$ is a \mSFT\ and $Y = \{ 0^\infty \}$, then $\cS(\phi) = \cS_0(\phi) = [0,h(X)] \cap \cP$, where $\cP$ is the set of all numbers $h \geq 0$ such that $e^h$ is a Perron number as in \S 1.

As relaxed restrictions are given in the construction of an \ito\ factor code between \rSFTs\ with unequal entropies \cite{Boy83, BoyT84, Jung11},
we present the following conjecture. Note that in the following $\cS(\phi)$ is always contained in $[h(Y),h(X)] \cap \cP$.

\begin{conj}
    Let $\pXtoY$ be a factor code between \mSofic s. Then $\cS(\phi) = [h(Y),h(X)] \cap \cP$. If $X$ is of finite type, then $\cS_0(\phi) = [h(Y),h(X)] \cap \cP$.
\end{conj}

\vspace{0.3cm}
\section{Decompositions of embeddings}

Let $\pXtoY$ be an embedding. If we construct a subshift $Z$ with $\phi(X) \subset Z \subset Y$, then this gives us a decomposition $\phi = \phi_2 \circ \phi_1$ of $\phi$ into embeddings where $\phi_2 : Z \to Y$ is the canonical embedding given by the inclusion and $\phi_1 : X \to Z$ is equal to $\phi$ except that the codomain of $\phi_1$ is $Z$. Thus finding a decomposition of an embedding into embeddings is closely related to the construction of subshifts lying between two shift spaces. As there are many results on constructing new shift spaces in \mSFTs, we already know more about decomposition of embeddings than that of factors.

\begin{prop}\cite[\S 26]{DGS} \label{prop:embedding_DGS}
    Let $\pXtoY$ be an embedding with $Y$ of finite type. Then $\cT_0(\phi)$ is dense in the interval $[h(X),h(Y)]$.
\end{prop}

In this section, we refine this result and give characterizations of the sets $\cT$, $\cT_0$ and $\cT_1$ for several cases. In contrast to decompositions of factor codes, it is easier to characterize these sets in finite-type case than in sofic case.
We recall some definitions and results. 

For a nonnegative integral square matrix $A$, we denote by $\X_A$ the \emph{edge shift} defined by $A$, i.e., the shift space which consists of all bi-infinite trips on the directed graph with the adjacency matrix $A$. A \SFT\ is conjugate to an edge shift. 
An edge shift $X$ is mixing \ifff\ $X = \X_A$ for some $A$ which is \emph{primitive}, i.e., $A^n > 0$ for some $n \in \setN$. For a shift space $X$, denote by $q_k(X)$ the number of periodic points of $X$ with least period $k$. If $X$ is a \mSFT, then $h(X) = \lim_{k \to \infty} k^{-1} \log q_k(X)$.

The following proposition can be found in \cite[Corollary 10.2]{AdlM} or \cite{Kri79}.

\begin{prop}\label{prop:subtuple_thm}
    Let $X$ be a \mSFT. Then there is a \mSFT\ $\widetilde X$ such that $\widetilde X$ has a fixed point and $h(\widetilde X) = h(X)$.
\end{prop}


The following lemma is referred to as Blowing up Lemma.
\begin{lem}\cite{Boy83} \label{prop:blowing_up}
    Let $X$ be an \rSFT\ with $h(X) > 0$ and $q_n(X) > 0$. Let $M_1, \cdots, M_k \geq 1$. Then there is an \rSFT\ $\widetilde X$ such that (1) $q_n(\widetilde X) = q_n(X) - n + n \cdot |\{ i : M_i = 1 \}|$, (2) $q_m(\widetilde X) = q_m(X) + n \cdot |\{ i : nM_i = m \}|$ if $m = nM_i$ for some $M_i > 1$, and (3) $q_i(\widetilde X) = q_i(X)$ for all other $i$.
    Furthermore, if $X$ is mixing then so is $\tilde X$.
\end{lem}

\begin{lem}\label{lem:extending_embedding_to_conjugacy} \cite{JungL}
    Let $X \subset \widetilde X$ be shift spaces and $\phi : X \to Y$ a conjugacy. Then there exist a shift space $\widetilde Y \supset Y$ and a conjugacy $\tilde \phi : \widetilde X \to \widetilde Y$ such that $\tilde \phi|_X = \phi$.
\end{lem}

We also need the following lemma equivalent to the extension theorem of Boyle, which, in turn, is a slight extension of Krieger's Embedding Theorem \cite{Kri82}.

\begin{lem}\cite[Lemma 2]{Boy88}\label{thm:Boyle_extension}
    Let $\phi$ be an embedding from a shift space $X$ into a \mSFT\ $Y$. If $Z$ is a shift space with $X \subset Z$, $h(Z) < h(Y)$ and $q_k(Z) \leq q_k(Y)$ for all $k \in \setN$ then there is an embedding $\tilde \phi : Z \to Y$ with $\tilde \phi|_X = \phi$.
\end{lem}

The following theorem is a key ingredient in characterizing the sets $\cT$ and $\cT_0$ for several cases.

\begin{thm}\label{thm:exists_subshift}
    Let $X \subsetneq Y$ be shift spaces with $Y$ mixing and sofic. Then for any $h \in (h(X),h(Y))$, there is a mixing shift space $Z$ such that $X \subset Z \subset Y$ and $h(Z) = h$. Furthermore, the following hold.
    \begin{enumerate}
        \item If $h \in \cP$, then $Z$ can be chosen to be sofic.
        \item If $h \in \cP$ and $Y$ is of finite type, then $Z$ can be chosen to be of finite type.
    \end{enumerate}
\end{thm}
\begin{proof}
    First, assume that $Y$ is a \SFT\ and $h \in \cP$, i.e., $e^h$ is a Perron number. Then there is a \mSFT\ $X_1$ such that $h(X_1) = \log e^h = h$. By Proposition \ref{prop:subtuple_thm}, one can find a \mSFT\ $X_2$ such that $X_2$ has a fixed point and $h(X_2) = h(X_1) = h$. Since $h(X) < h(X_2)$, it follows that $q_n(X) < q_n(X_{2})$ for all sufficiently large $n \in \setN$. Thus by applying Blowing up lemma to a fixed point in $X_{2}$, we can find a \mSFT\ $W_1$ such that $h(W_1) = h(X_{2}) = h$ and $q_n(X) \leq q_n(W_1)$ for all $n \in \setN$. 

    Since $h(W_1) < h(Y)$, we have $q_n(W_1) < q_n(Y)$ for all large $n \in \setN$. By applying Blowing up lemma repeatedly to points in $W_1$ having low periods, it is possible to obtain a \mSFT\ $W_2$ such that $h(W_2) = h$ and
    \[ q_n(X) \leq q_n(W_2) \leq q_n(Y) \text{ for all $n \in \setN$}.\]
    (We remark that this is an argument appeared in \cite[\S 2]{Boy83}.)

    By Krieger's Embedding Theorem \cite{Kri82}, there is an embedding from $X$ into $W_2$, and by Lemma \ref{lem:extending_embedding_to_conjugacy} there is a \mSFT\ $W$ such that $W$ is conjugate to $W_2$ and $X \subset W$. Since $Y$ is of finite type, we can extend the inclusion map $i: X \to Y$ to an embedding $\bar i: W \to Y$ using Lemma \ref{thm:Boyle_extension}. Then $Z = \bar i(W)$ is the desired subshift, which proves (2).

    \vspace{0.1cm}
    Next, assume that $h \in \cP$ and $Y$ is merely sofic. Let $\pi : \widetilde Y \to Y$ be a factor code such that $\widetilde Y$ is a mixing \SFT\ and $h(\widetilde Y) = h(Y)$ (e.g., a code given by its minimal right resolving presentation \cite{LM}).
    Let $\widetilde X = \pi^{-1}(X)$. Since $h(\widetilde X) = h(X)$, by (2) we can find a \mSFT\ $\widetilde Z$ such that $\widetilde X \subset \widetilde Z \subset \widetilde Y$ and $h(\widetilde Z) = h$. Then $Z = \pi(\widetilde Z)$ satisfies all the conditions, which proves (1).

    \vspace{0.1cm}
    Finally, we prove the general case. Since $\cP$ is dense in $[0,\infty)$, there are a strictly increasing sequence $\{ \lambda_i \}_{i \in \setN}$ in $\cP$ and a strictly decreasing sequence $\{ \eta_i \}_{i \in \setN}$ in $\cP$ such that $\lambda_i \to h$ and $\eta_i \to h$. Without loss of generality we can assume that $h(X) < \lambda_1 < \eta_1 < h(Y)$. Let $X_0 = X$ and $Y_0 = Y$. For each $n \in \setN$, define $X_n$ and $Y_n$ inductively as follows: Let $X_n$ be a mixing \Sofic\ such that $X_{n-1} \subset X_n \subset Y_{n-1}$ and $h(X_n) = \lambda_n$. Also let $Y_n$ be a \mSofic\ such that $X_{n} \subset Y_{n} \subset Y_{n-1}$ and $h(Y_n) = \eta_n$. Note that (1) guarantees the existence of $X_n$ and $Y_n$ for each $n \in \setN$.

    Now let $Z$ be the closure of the union of $X_n$, $n \in \setN$. Since $\{X_n\}_{n \in \setN}$ is an increasing sequence of mixing shift spaces, $Z$ is also a mixing shift space (which follows from the fact $\B(Z) = \bigcup_n \B(X_n)$). Also we have $h(Z) \geq \lim_{n \to \infty} h(X_n) = h$. Since $Z \subset Y_n$ for all $n \in \setN$, it follows that $h(Z) = h$, which completes the proof.
\end{proof}

\begin{rem}
    A shift space $X$ is called a \emph{coded system} if it contains an increasing sequence of \rSFTs\ whose union is dense in $X$. In fact, one can take $Z$ to be a coded system in Theorem \ref{thm:exists_subshift}.
\end{rem}

The following is just a restatement of Theorem \ref{thm:exists_subshift} (apply the argument in the beginning of this section). Note that $h(Y)$ is clearly contained in the sets $\cT(\phi), \cT_0(\phi)$ and $\cT_1(\phi)$.

\begin{cor}\label{cor:cT_sofic}
    Let $\pXtoY$ be an embedding into a mixing sofic shift $Y$. Then $\cT(\phi) \supset (h(X),h(Y)]$ and $\cT_1(\phi) \supset (h(X),h(Y)] \cap \cP$. If $Y$ is of finite type, then $\cT_0(\phi) \supset (h(X),h(Y)] \cap \cP$.
\end{cor}

We now give characterizations of $\cT$ where $Y$ is sofic, and $\cT_0$ where $X$ and $Y$ are of finite type. By reducing from the irreducible sofic case to the mixing finite type case, the following corollaries are immediate.

\begin{cor}\label{cor:cT_sofic}
    Let $\pXtoY$ be an embedding into an irreducible sofic shift. Then $\cT'(\phi) = (h(X),h(Y)]$. If $X$ is irreducible, then $\cT(\phi) = [h(X),h(Y)]$.
\end{cor}

\begin{cor}\label{cor:cT_0_SFT}
    Let $\pXtoY$ be an embedding between \rSFTs\ $X$ and $Y$ with periods $p$ and $q$, respectively. Then
    \[
        \cT_0(\phi) = [h(X),h(Y)] \cap \{ h \in \setR : r\cdot h \in \cP \text{ for some } r \in \setN \text{ with } q|r \text { and } r|p \}.
    \]
\end{cor}

\begin{rem}
    Let $\pXtoY$ be an embedding from a nonwandering \SFT\ $X$ into an \rSFT\ $Y$. Then for each $\epsilon > 0$ one can find an \rSFT\ $\widetilde X$ such that the period of $\widetilde X$ is equal to that of $X$, $\phi(X) \subset \widetilde X \subset Y$, and $h(\widetilde X)$ is $\epsilon$-close to $h(X)$. Thus as in Corollary \ref{cor:cT_0_SFT} we have
    \[
        \cT'_0(\phi) = (h(X),h(Y)] \cap \{ h \in \setR : r\cdot h \in \cP \text{ for some } r \in \setN \text{ with } q|r \text { and } r|p \}.
    \]
    However this equality does not hold if $X$ is merely of finite type. For a simple example, let $Y = \{ 0, 1, 2, 3, 4, 5, 6 \}^\setZ$ be the full 7-shift and $X$ the \SFT\ defined by $X = \{ \sigma^k(x) : x \in A \text { and } k \in \setZ \}$, where
    \[
        A = \{ (10)^\infty, (23)^\infty, (10)^\infty.4(23)^\infty, (10)^\infty.56(23)^\infty \}.
    \]
\end{rem}

\vspace{0.15cm}
To prove the case of sofic shifts, we need the following lemmata. In what follows, a \emph{\bic} code is a code which does not collapse two distinct right (or left) asymptotic points. A \bic\ code is topologically conjugate to a \emph{\bir} code, which is well studied in symbolic dynamics. Also, a code is \bic\ \ifff\ it uniformly separates the fibers, i.e., there is an $\epsilon > 0$ such that each preimage of a point in the range is an $\epsilon$-separated set (e.g., see \cite[Lemma 2.6]{Jung09}). For more on closing codes and \bic\ codes, see \cite{Jung09, KitMT91, Nas83}.

\begin{lem}\cite[Lemma 2.4]{BoyK88} \label{lem:BoyK}
    Let $X \subset \widetilde X$ be \SFTs\ and $\pi : X \to Y$ a \bic\ factor code. Then there exist a shift space $\widetilde Y \supset Y$ and a \bic\ factor code $\tilde \pi : \widetilde X \to \widetilde Y$ extending $\pi$ such that for any $x \in \widetilde X \setminus X$, $\tilde \pi(x)$ has a unique preimage.
\end{lem}

\begin{lem}\cite[Theorem 26.17]{DGS} \label{lem:DGS_2}
    Let $X$ be a mixing \SFT\ and $Z$ a proper subshift of $X$. Then for given $h < h(X)$, there is a mixing \SFT\ $V \subset X$ such that $h(V) > h$ and $V \cap Z = \emptyset$.
\end{lem}

\begin{thm}\label{thm:exists_sofic_with_weakPerron}
    Let $X$ be a shift space and $Y$ a \mSFT\ with $X \subsetneq Y$. Then for $h \in (h(X),h(Y)) \cap \cP^w$, there is an \rSofic\ $Z$ with 
    \[ X \subset Z \subset Y \text{ and } h(Z) = h. \]
\end{thm}
\begin{proof}
    By Theorem \ref{thm:exists_subshift} (2), there is a \mSFT\ $\widetilde X$ such that $X \subset \widetilde X \subset Y$ and $h(\widetilde X) < h$. By passing to a higher block shift, we may assume that $\widetilde X = \X_B$ for a primitive matrix $B$. Also, by Lemma \ref{lem:DGS_2}, there is a \mSFT\ $V \subset Y$ disjoint from $\widetilde X$ with $h(V) > h$.

    Let $m \in \setN$ be the size of $B$. For each $n \in \setN$ define an $nm \times nm$ matrix $B_n$ by
        $$B_n =
        \begin{pmatrix}
            0 & B & 0 & \cdots & 0 \\
            0 & 0 & B & \cdots & 0 \\
            \vdots & \vdots & \vdots & \ddots & \vdots \\
            0 & 0 & 0 & \cdots & B \\
            B & 0 & 0 & \cdots & 0
        \end{pmatrix}.$$
    Note that $B_n$ is irreducible and $\per(B_n) = n$ for each $n \in \setN$. Since
    \[ h(\X_{B_n}) = h(\X_B) < h < h(V),\]
    by the proof of \cite[Lemma 4.1]{Jung11} the shift space $\X_{B_n}$ embeds into $V$ for all large $n$. Take $n \in \setN$ large so that $n \cdot h \in \cP$ and $\X_{B_n}$ embeds into $V$. For brevity, we also denote by $\X_{B_n}$ the embedded image in $V$.

    Then by Corollary \ref{cor:cT_0_SFT}, there is an \rSFT\ $W$ such that $\X_{B_n} \subset W \subset V$ and $h(W) = h$.
    Note that the natural projection code $\pi$ from $\X_{B_n}$ onto $\X_B$ is \bir\ (e.g., see \cite{Nas83}). Thus by Lemma \ref{lem:BoyK}, we can find a (sofic) shift space $Z_1 \supset \X_{B}$ and a \bic\ factor code $\tilde \pi : W \to Z_1$ extending $\pi$. Since \bic\ codes preserve entropy, we have $h(Z_1) = h(W) < h(Y)$.
    Note that $\tilde \pi$ is 1-1 from $W \setminus \X_{B_n}$ onto $Z_1 \setminus \X_{B}$.
    Also, $W \cap \X_B \subset V \cap \X_B = \emptyset$. It follows that
    \[ q_k (Z_1) \leq q_k(W) + q_k(\X_B) \leq q_k(Y) \]
    for each $k \in \setN$. Thus we can extend an embedding $i : \X_B \to Y$ (given by the inclusion in the beginning of the proof) to an embedding $\bar i : Z_1 \to Y$ by Lemma \ref{thm:Boyle_extension}. The sofic shift $Z = \bar i (Z_1)$ is a desired one (indeed, this is an AFT shift \cite{BoyKM85} since $\tilde \pi$ is \bic).
\end{proof}

The following result follows from the reduction to the finite type and mixing case.
\begin{cor}\label{cor:cT_1_sofic}
    Let $\phi: X \to Y$ be an embedding into an irreducible sofic shift $Y$. Then $\cT'_1(\phi) = (h(X),h(Y)] \cap \cP^w$. If $X$ is irreducible and sofic, then $\cT_1(\phi) = [h(X),h(Y)] \cap \cP^w$.
\end{cor}

\begin{rem}
    In general, $\cT_0(\phi)$ is not even dense under the assumption of Corollary \ref{cor:cT_1_sofic}. For a simple example, let $Y$ be the even shift \cite[\S 1.2]{LM} and $X = \{ 0^\infty \}$. If $Z$ is a \SFT\ which contains $X$ and is contained in $Y$, then $Z = X$. Thus $\cT_0(\phi)$, where $\phi : X \to Y$ is the inclusion map, consists only of one point (compare with Proposition \ref{prop:embedding_DGS}).
\end{rem}

\begin{proof}[Proof of Theorem \ref{thm:embedding_introduction}]
    The result follows from Corollary \ref{cor:cT_sofic}, Corollary \ref{cor:cT_0_SFT}, and Corollary \ref{cor:cT_1_sofic}.
\end{proof}

\vspace{0.1cm}
\begin{ack*}
We thank Sujin Shin for her good advice. We also thank the referee for useful suggestions and comments, especially for pointing out the result stated in Remark \ref{rem:decompose_factor_sofic_case} (2).
The first named author was supported by Basal Grant-CMM and Fondap 15090007. The second named author was supported by TJ Park Postdoctoral Fellowship.
\end{ack*}
\vspace{0.4cm}

\bibliographystyle{amsplain}
\bibliography{_Bib_Decomposition}

\end{document}